\documentclass[notitlepage]{revtex4-1}
\usepackage{amsmath, amsthm}%
\usepackage{amsfonts}%
\usepackage{amssymb}%
\usepackage{graphicx,hyperref}
\theoremstyle{plain}

\newtheorem{lemma}{Lemma}

\numberwithin{equation}{section}
\newcommand{\arctanh}{\operatorname{arctanh}}
\newcommand{\R}{\mathbb{R}}

\providecommand{\one}{\leavevmode\hbox{\small1\kern-3.8pt\normalsize1}}

\begin{document}

\title[Short Title (for the running head)]{A note on ``The Cartan-Hadamard conjecture and the Little Prince"}
\author{S. Michalakis}
\affiliation{Institute for Quantum Information and Matter, Caltech, Pasadena, CA 91125}

\begin{abstract}
We provide elementary proofs of Lemmas 7.1 and 7.4 appearing in "The Cartan-Hadamard conjecture and the Little Prince", by B. Kloeckner and G. Kuperberg. The Lemmas play an important role in the derivation of novel isoperimetric inequalities. The original proofs relied on Sage, a symbolic algebra package, to factor certain algebraic varieties into irreducible components.
\end{abstract}
\maketitle

\section{The two lemmas}
The following lemmas can be found in~\cite{little-prince} as Lemmas 7.1 and 7.4. Elementary versions of their proofs follow below.
\begin{lemma}[Lemma 7.1]
Show that for $x,y \ge 0$ and $\theta \in [0,\pi]$,
\begin{equation}
(\sin\theta)^3 xy + \left((\cos\theta)^3 -3\cos\theta +2\right) (x+y) - (\sin\theta)^3-6\sin\theta -6\theta + 6\pi - 6\arctan(x) +\dfrac{2x}{1+x^2} -6\arctan(y) +\dfrac{2y}{1+y^2} \ge 0,
\end{equation} 
with equality at $x=y=\cot(\theta/2)$.
\end{lemma}
\begin{proof}
We begin by setting 
\begin{multline}
G(\theta,x,y) = (\sin\theta)^3 xy + \left((\cos\theta)^3 -3\cos\theta +2\right) (x+y) - (\sin\theta)^3-6\sin\theta -6\theta + 6\pi - 6\arctan(x) +\dfrac{2x}{1+x^2} \\ -6\arctan(y) +\dfrac{2y}{1+y^2}.
\end{multline}
In the proof of Lemma 7.1 in~\cite{little-prince}, the authors begin by showing that the negative values of $G(\theta,x,y)$ are confined in a compact subset of $[0,2\pi]\times \R^2$. Hence, we can use the derivative and boundary-value test to prove positivity.

It is easy to see that $G(0,x,y) = 6\pi - 6\arctan(x) +\dfrac{2x}{1+x^2} -6\arctan(y) +\dfrac{2y}{1+y^2} \ge 0$, with equality only when $x = y = \cot(0) = \infty$, so we may assume that $\theta > 0$.
Noting that the following identity holds (by differentiating both sides):
\[4x - 6\arctan(x) +2x/(1+x^2) = 4 \int_0^x \frac{s^4}{(1+s^2)^2} ds,\]
(and similarly for $y$), we see that the condition $G(\theta,x,y) \ge 0$, is equivalent to:
\begin{equation}\label{eqn:new_version}
(\sin\theta)^3 xy + ((\cos\theta)^3 -3\cos\theta -2) (x+y) - (\sin\theta)^3-6\sin\theta -6\theta + 6\pi + 4 \int_0^x \frac{s^4}{(1+s^2)^2} ds+4 \int_0^y \frac{s^4}{(1+s^2)^2} ds \ge 0.
\end{equation}
Setting the partial derivatives of $G(\theta,x,y)$ to zero, yields in turn:
\begin{eqnarray}\label{eqn:partial_theta}
\partial_\theta G(\theta,x,y) &=& 0 \implies \sin^2(\theta) (\cos(\theta) xy + \sin(\theta)(x+y)) = 2 (1+\cos(\theta))+\sin^2(\theta)\cos(\theta),\\ \label{eqn:partial_x}
\partial_x G(\theta,x,y) &=& 0 \implies \sin^3(\theta) y + 4x^4/(1+x^2)^2 = (1+\cos\theta)^2 + \sin^2\theta (1+\cos\theta),\\ \label{eqn:partial_y}
\partial_y G(\theta,x,y) &=& 0 \implies \sin^3(\theta) x + 4y^4/(1+y^2)^2 = (1+\cos\theta)^2 + \sin^2\theta (1+\cos\theta),
\end{eqnarray}
where we used $(\cos\theta)^3 -3\cos\theta -2 = - (1+\cos\theta)^2(2-\cos\theta) = -(1+\cos\theta)^2 - \sin^2\theta (1+\cos\theta),$ to get the last two equations. 

Now, noting that $2(1+\cos(\theta))+\sin^2(\theta)\cos(\theta) = \sin^2(\theta) \left(\cos\theta \cot^{2}(\theta/2) + 2\sin\theta \cot(\theta/2)\right),$ where we used, $\cot(\theta/2) = (1+\cos\theta)/\sin(\theta)$, we get from Eqn~\eqref{eqn:partial_theta}:
$$ \partial_\theta G(\theta,x,y) = 0 \implies \sin^2(\theta) \left[\cos(\theta) (xy-\cot^{2}(\theta/2)) + \sin(\theta)(x+y-2\cot(\theta/2))\right] = 0.$$
Hence, $\theta = \pi$, or 
\begin{equation}\label{eqn:theta_cond}
\cos(\theta) (\cot^2(\theta/2)-xy) = \sin(\theta)(x+y-2\cot(\theta/2)), \,\theta \in (0,\pi).
\end{equation}
A quick check shows that:
$G(\pi,x,y) = 4 \int_0^x \frac{s^4}{(1+s^2)^2} ds+4 \int_0^y \frac{s^4}{(1+s^2)^2} ds \ge 0,$
since $x,y \ge 0$, with equality when $x=y=\cot(\pi/2)=0$. To treat the remaining case, we make the following substitutions:
\begin{equation}\label{eqn:subs}
x= \alpha \cot(\theta/2), \quad y = \beta \cot(\theta/2).
\end{equation}

Substituting the above in equation~\eqref{eqn:theta_cond}, and using $\cot(\theta/2) = (1+\cos\theta)/\sin\theta$, we get:
$$(1-\alpha\beta) \cos\theta = ((\alpha-1) + (\beta-1)) (1-\cos\theta) ,$$
which can be further simplified to:
$$(1-\alpha)(1-\beta) \cos\theta = (1-\alpha) + (1-\beta).$$
Now, if $\alpha=1$, then $\beta =1$ (and vice versa), and equation~\eqref{eqn:subs} implies that $x=y=\cot(\theta/2)$, which is what we set out to prove. So, we may assume that $(1-\alpha)(1-\beta) \neq 0$ and that:
$$\cos\theta = (1-\alpha)^{-1} + (1-\beta)^{-1}, \quad \theta \in (0,\pi).$$
The rest of the proof verifies that the above conditions on $\alpha, \beta$ and $\theta$ are incompatible with \eqref{eqn:partial_x} and \eqref{eqn:partial_y}, so that the only global minimum is attained at $x=y=\cot(\theta/2)$, at which point $G(\theta,x,y) = 0$.

We begin with \eqref{eqn:partial_x}. Using equation~\eqref{eqn:subs} and the trigonometric identities $\cot(\theta/2) = (1+\cos\theta)/\sin\theta$ and $ \cot^2(\theta/2) = (1+\cos\theta)/(1-\cos\theta)$, we see that $\partial_x G(\theta,x,y) = 0$, is equivalent to the following equality:
\begin{equation}
4\alpha^4 =\left((\alpha^2-1)\cos\theta+\alpha^2+1\right)^2 \left(1 + (1-\beta)(1-\cos\theta)\right).
\end{equation}
Substituting the expression we derived above for $\cos\theta$, we get:
\begin{equation}\label{eqn:alpha0}
4\alpha^4 (1-\beta) = \alpha (\alpha-1) (1+\alpha + \alpha(1-\beta))^2,
\end{equation}
where we used $1 + (1-\beta)(1-\cos\theta) = \alpha (1-\beta) (\alpha-1)^{-1}$ and
$(\alpha^2-1)\cos\theta+\alpha^2+1 = (\alpha+1)((\alpha-1)\cos\theta+1)+\alpha(\alpha-1) = (\alpha-1)(1-\beta)^{-1} (2\alpha + 1-\alpha\beta).$

If $\alpha = 0$, then equation \eqref{eqn:alpha0} with $\alpha$ and $\beta$ exchanged (derived by taking the partial w.r.t. $y$ instead of $x$, and noting that $G(\theta,x,y) = G(\theta,y,x)$) yields $4\beta^4 = \beta(\beta-1)(1+2\beta)^2$, which is equivalent to $\beta (3\beta+1) = 0$. Since $\beta \ge 0$, we have that $\alpha = 0$ implies $\beta = 0$, which leads to the contradiction $\cos\theta = 2$ (note that $x=y=0$ minimizes $G(\theta,x,y)$ when $\theta=\pi$, which was treated separately above). By symmetry, we also know that if $\beta = 0$, then $\alpha = 0$, since otherwise $3\alpha+1 = 0$, which is a contradiction. Hence, we can assume that $\alpha, \beta > 0$.

Dividing through by $\alpha$ and re-writing $4\alpha^3(1-\beta) = 4\alpha(1+\alpha)(\alpha-1)(1-\beta) + 4\alpha(1-\beta)$, equation \eqref{eqn:alpha0} yields:
\begin{equation}\label{eqn:alpha}
4\alpha (1-\beta) = (\alpha-1)(1+\alpha\beta)^2,
\end{equation}
which may be further simplified to:
$$4\alpha +(1-\alpha\beta)^2 = \alpha (1+\alpha\beta)^2.$$
Moreover, by exchanging $\alpha$ and $\beta$ (due to the symmetry of $G$ discussed above), we also have the equation:
$$4\beta +(1-\alpha\beta)^2 = \beta (1+\alpha\beta)^2.$$
Subtracting the two equations yields:
$$4(\alpha-\beta) = (\alpha-\beta)(1+\alpha\beta)^2,$$
which implies that $ \alpha = \beta$ and/or $ \alpha\beta =1$. If $\alpha = \beta$, the \eqref{eqn:alpha} becomes 
$$ 4\alpha (1-\alpha) = (\alpha-1)(1+\alpha^2)^2,$$ 
which immediately implies $\alpha = 1$ (since the left and right side of the equality have opposite signs otherwise). This contradicts our earlier assumption that $(1-\alpha)(1-\beta) \neq 0$. If $\alpha\beta = 1$, then $\cos\theta = (1-\alpha)^{-1}+(1-1/\alpha)^{-1} = 1$, which contradicts the assumption that $\theta \in (0,\pi)$. This completes the proof.
\end{proof}

\begin{lemma}[Lemma 7.4]
Let
\begin{multline}
F(\ell,x,y) = \sinh(\ell)^3xy - \left(\cosh(\ell)^3-3\cosh(\ell)+2\right)(x+y)
    + \sinh(\ell)^3 - 6\sinh(\ell) - 6\ell + 6\arctanh\left(\frac1x\right)
    + \frac{2x}{x^2-1} \\
    + 6\arctanh\left(\frac1y\right) + \frac{2y}{y^2-1}.
\label{e:Fn4} \end{multline}
The function $F(\ell,x,y)$ is non-negative for $\ell \ge 0$ and $x,y \ge 1$, vanishing only when $x = y = \coth(\ell/2)$.
\end{lemma}
\begin{proof}
The proof follows the steps of the argument given above for Lemma 1 almost identically (the argument confining the negative values of $F$ in a compact set can be found in the beginning of the proof of Lemma 7.4 in~\cite{little-prince}). In particular, applying similar reasoning as before and using $\coth(\ell) = (1+\cosh\ell)/\sinh\ell$, one can show that:
\begin{eqnarray}\label{eqn:partial_ell}
\partial_\ell F(\ell,x,y) &=& 0 \implies \sinh^2(\ell) \left[\cosh(\ell) (xy-\coth^2(\ell)) - \sinh(\ell)(x+y-2\coth(\ell))\right] = 0,\\ \label{eqn:partial_x2}
\partial_x F(\ell,x,y) &=& 0 \implies \sinh^3(\ell) y - 4x^4/(x^2-1)^2 = -(1+\cosh\ell)^2 + \sinh^2\ell (1+\cosh\ell),
\end{eqnarray}
and similarly for $\partial_y F(\ell,x,y)$, since $F(\ell,x,y)$ is symmetric in $x$ and $y$. It is trivial to check that $F(0,x,y) \ge 0$, with equality only when $x=y=\coth(0)=\infty$. Hence, we will assume that $\ell > 0$ from now on. Setting $x = \alpha \coth(\ell/2), y = \beta\coth(\ell/2)$ and using $\cosh^2(\ell)-\sinh^2(\ell)=1$, we get:
\begin{eqnarray}
\partial_\ell F(\ell,x,y) &=& 0 \implies (1-\alpha\beta) \cosh(\ell) = \left((\alpha-1) + (\beta -1)\right) (1-\cosh\ell),\\
\partial_x F(\ell,x,y) &=& 0 \implies 4\alpha^4 =\left((\alpha^2-1)\cosh\ell+\alpha^2+1\right)^2 \left(1 + (1-\beta)(1-\cosh\ell)\right),\\
\partial_y F(\ell,x,y) &=& 0 \implies 4\beta^4 =\left((\beta^2-1)\cosh\ell+\beta^2+1\right)^2 \left(1 + (1-\alpha)(1-\cosh\ell)\right).
\end{eqnarray}
Noting that the above conditions for extrema of $F(\ell,x,y)$ are identical to the ones we derived for $G(\theta,x,y)$ in the proof of Lemma 1 (if we swap $\cos\theta$ with $\cosh\ell$), we see that the only values of $\alpha$ and $\beta$ satisfying all three conditions are $\alpha = \beta =1$ (one still needs to check that $\alpha = \beta = 0$ is not allowed, which follows from the requirement that $x,y \ge 1$). Hence, the global minimum of $F(\ell,x,y)$ is attained at $x=y=\cot(\ell/2)$, at which point it is each to check that $F(\ell,x,y)=0$.
\end{proof}

\acknowledgments
S. Michalakis would like to acknowledge support provided by the Institute for Quantum Information and Matter, an NSF Physics Frontiers Center with support of the Gordon and Betty Moore Foundation through Grant \#GBMF1250 and by the AFOSR Grant \#FA8750-12-2-0308.


\begin{thebibliography}{99}
\bibitem{little-prince} B. Kloeckner and G. Kuperberg, ``The Cartan-Hadamard conjecture and The Little Prince'', \href{http://arxiv.org/abs/1303.3115}{arXiv:1303.3115 \bf{[math.DG]}} (2014)
\end{thebibliography}
\end{document}